\documentclass
{amsart}
\usepackage{amsmath,amssymb,amscd,amsfonts,amsmath}

\newcommand{\C}{{\mathbf C}}                   
\newcommand{\Z}{{\mathbf Z}}                   
\newcommand{\A}{{\mathbf A}^1}                   
\renewcommand{\P}{\mathbf{P}^1}                
\renewcommand{\d}{\mbox{d}}                      
\newcommand{\Mod}{{\mathcal M}}

\DeclareMathOperator{\res}{res} 

\newtheorem{prop}{Proposition}[]
\newtheorem{prop2}{Proposition}[]

\newtheorem{rk}[prop]{Remark}

\newtheorem{lem}[prop]{Lemma}

\newtheorem{cor}[prop2]{Corollary}
\newtheorem{thm}[prop]{Theorem}

\title[The dimension of Garnier equations]
{The dimension of the space of Garnier equations with fixed locus of apparent singularities}
\author[Szil\'ard Szab\'o]{Szil\'ard Szab\'o} 
\thanks{
  Department of Geometry, 
  Budapest University of Technology and Economics, 
  Egry J. u. 1/H,
  Budapest 1111, Hungary,
  \texttt{szabosz@math.bme.hu}
}
\date{\today}

\begin{document}

\begin{abstract}
We show that the conditions imposed on a second order linear differential equation 
with rational coefficients on the complex line by requiring it to have regular singularities 
with fixed exponents at the points of a finite set $P$ and apparent singularities 
at a finite set $Q$ (disjoint from $P$) determine a linear system of maximal rank. 
In addition, we show that certain auxiliary parameters can also be fixed. 
This enables us to conclude that the family of such differential equations is of the expected dimension
and to define a birational map between an open subset of the moduli space of logarithmic connections 
with fixed logarithmic points and regular semi-simple residues and the Hilbert scheme of points on 
a quasi-projective surface.
\end{abstract}

\maketitle

\section{Introduction}
The fact that some singularities of scalar linear differential equations with meromorphic coefficients on a 
complex domain are "inessential" was already observed by Weierstrass. 
In recent times the study of these apparent singularities of linear differential equations has attracted 
growing interest \cite{vdps}, \cite{iis}, \cite{dubma}, \cite{ss}. 
A basic question in this field is to determine the dimension of the space of 
Fuchsian differential equations having fixed exponents at a fixed set $P$ and apparent singularities 
at a disjoint finite set $Q$ with prescribed ``multiplicities'' (or weights) of the points. 
The formal dimension count is easy to do and is carried out for example in Remark 6.3 of \cite{vdps}. 
However, it is also stated there that it is not yet known whether the conditions imposed on the parameters 
are independent from each other (except for one obvious linear relation between them). 
In this article, we first give an affirmative answer to this question for differential equations of order $2$ 
assuming the total weight of the apparent singularities to be the expected value (i.e. half the dimension of the moduli space). 

Of course, in case the total weight of $Q$ is smaller than $n-2$ the result still holds, hence showing that the 
dimension of the corresponding space of linear differential equations is the same as for the total weight 
$n-2$ case. If however $Q$ is of total weight greater than $n-2$ then the linear system to solve is overdetermined. 
It turns out that in this case the existence of such connections imposes quadratic constraints on the natural 
additional parameters of such an equation. 

As a consequence, in Corollary \ref{cor} we show that for any finite set $Q$ of arbitrary total weight 
which is disjoint from $P$ the dimension of the space of second-order linear differential equations 
having fixed eigenvalues of its residues at $P$ and apparent singularities of 
prescribed weights at the points of $Q$ is $n-2$. 

The relevance of this result lies in the study of the fibration of an open subset of the corresponding 
moduli space of stable parabolic logarithmic connections over an open subset of the space $\mathbf{CP}^{n-2}$ of all 
positions $Q$ of apparent singularities. Namely, it implies that all the fibers of this fibration are 
of dimension $n-2$ (which is moreover equal to half the dimension of the whole moduli space). 
We explain this geometric interpretation in Corollary \ref{cor:geomint} and Section 
\ref{sec:geomint}. 

We are aware that a similar result is also proved in Section 3.4.3. of \cite{iksy}; 
that proof however uses methods only adapted to the second-order case. 
Indeed, in the terminology of Section \ref{sec:cor} below, their proof heavily relies on the identity $N=n-2$ 
since they consider an invertible matrix of size $N\times (n-2)$. However, this identity can only be 
expected to hold for second-order equations (and even there only generically in the moduli space). 
Furthermore, the work \cite{dubma} contains computations in direction of the general case of Fuchsian 
equations of arbitrary order, however mostly in the weight $1$ case (called special in {\it loc. cit.})
Therefore, our primary aim here is to give a simple and self-contained proof of this statement 
covering the higher-weight case too and which lends itself to direct (although far from obvious) 
generalisation to higher-order equations; we will adress this generalisation in future work. 

During the preparation of this work the author benefited from support of Lend\"ulet project 
and OTKA grant NK 81203. 

\section{Statement of the result}
Denote the Riemann sphere by $\P$, let $\A\subset\P$ be an affine patch and $z$ be a coordinate on $\A$, 
and set $\P\setminus\A=\infty$. 
Fix $n\geq 2$ and a finite subset $P=\{t_0,t_1,\ldots,t_n\}$ of $\P$. We will assume $t_0=\infty$ and 
$t_{i_1}\neq t_{i_2}$ for $i_1\neq i_2$. 
Furthermore, fix a set $Q=\{q_1,\ldots,q_N\}\subset\P$ and for each 
$j\in\{1,\ldots,N\}$ a positive integer $w_j$ (the \emph{weight} of $q_j$) 
satisfying 
\begin{equation}\label{sumw}
        \sum_{j=1}^N w_j = n-2. 
\end{equation}
The sum on the left-hand side is called the \emph{total weight} $w(Q)$ of $Q$. 
Again, we assume that the points $q_j$ are pairwise different from each other; 
however, in general we do not assume the sets $P$ and $Q$ to be disjoint. Set 
\begin{equation}
	\psi(z)=\prod_{i=1}^n(z-t_i)\prod_{j=1}^N(z-q_j)
\end{equation}
and consider a second order linear differential equation with rational coefficients 
\begin{equation}\label{eq}
	w''(z) + \frac{G(z)}{\psi(z)}w'(z) + \frac{H(z)}{\psi(z)^2}w(z) = 0 
\end{equation}
for the holomorphic function $w(z)$ and its first and second holomorphic differentials $w'$ and $w''$, 
where $G$ and $H$ are polynomials of degree $2n-3$ and $4n-6$ respectively: 
\begin{align}
	G(z) &= G_0 + G_1 z + \cdots + G_{2n-3} z^{2n-3} \label{g}\\
	H(z) &= H_0 + H_1 z + \cdots + H_{4n-6} z^{4n-6}. \label{h}
\end{align}
It is well-known that such an equation has regular singularities at $P\cup Q$. 
For all $i\in\{1,\ldots,n\}$ set 
\begin{align}
	G^{t_i}_0 &= \res\left(\frac{G(z)}{\psi(z)},z=t_i\right) 
		= \lim_{z\to t_i}\left((z-t_i)\frac{G(z)}{\psi(z)}\right),\label{gi}\\
	H^{t_i}_0 &= \lim_{z\to t_i}\left((z-t_i)^2\frac{H(z)}{\psi(z)^2}\right)\label{hi}
\end{align}
for the lowest non-necessarily zero coefficients of the Laurent-series near $t_i$ of the coefficients 
in (\ref{eq}) of $w'$ and $w$ respectively. In a similar vein, we set 
\begin{equation}\label{g0h0}
	G^{t_0}_0 = G_{2n-3}, H^{t_0}_0 = H_{4n-6}
\end{equation}
for the lowest non-necessarily zero coefficients of the Laurent-series near $t_0$ of the same coefficients. 
For $i\in\{0,\ldots,n\}$ the exponents of (\ref{eq}) at $t_i$ are defined as the roots $\rho^i_1,\rho^i_2$ of 
the indicial equation 
$$
	\rho(\rho-1) + G^{t_i}_0 \rho + H^{t_i}_0=0; 
$$
then the local monodromy of (\ref{eq}) over a loop going around $t_i$ once in the positive 
direction and having no other points of $P\cup Q$ in its interior has eigenvalues 
$\exp(2\sqrt{-1}\pi\rho^i_1),\exp(2\sqrt{-1}\pi\rho^i_2)$. 
For any $i\in\{0,\ldots,n\}$, fix $\rho^i_1,\rho^i_2\in\C$ in such a way that 
the following two conditions hold: 
\begin{enumerate}
\item $\rho^i_1-\rho^i_2\notin\Z$, \label{condgeni}
\item for any choice of $\varepsilon_i\in\{1,2 \}$ for all $i\in\{0,\ldots,n\}$ 
the sum $\sum_{i=0}^n\rho^i_{\varepsilon_i}$ is not an integer. \label{condgenii}
\end{enumerate}
We will refer to these conditions as \emph{genericity} of the eigenvalues. 
Condition (\ref{condgeni}) implies that the local monodromy about $t_i$ 
is regular semisimple. Condition (\ref{condgenii}) implies that a logarithmic 
connection with eigenvalues $\rho^i_1,\rho^i_2$ at $t_i$ has no rank $1$ 
sub-bundle invariant by the connection; indeed, the residue of such a subbundle 
at $t_i$ would have to be one of $\rho^i_1,\rho^i_2$, and the degree of the line 
bundle would have to be equal to minus the sum of these eigenvalues $\rho^i_{\varepsilon_i}$. 
Clearly, fixing $\rho^i_1,\rho^i_2$ is equivalent to fixing 
\begin{align}
	G^{t_i}_0 &= 1-(\rho^i_1+\rho^i_2),\label{condgi}\\
	H^{t_i}_0 &= \rho^i_1\rho^i_2.\label{condhi}
\end{align}
Similarly to (\ref{gi},\ref{hi}), for any $j\in\{1,\ldots,N\}$ we set 
\begin{align}
	G^{q_j}_0 &= \res\left(\frac{G(z)}{\psi(z)},z=q_j\right) 
		= \lim_{z\to q_j}\left((z-q_j)\frac{G(z)}{\psi(z)}\right),\label{gj}\\
	H^{q_j}_0 &= \lim_{z\to q_j}\left((z-q_j)^2\frac{H(z)}{\psi(z)^2}\right)\label{hj}. 
\end{align}
A point $q\notin P$ is called an apparent singularity of weight $w$ of (\ref{eq}) 
if near $q$ a fundamental system consisting of regular functions $w_1(z),w_2(z)$ 
can be found such that $w_1(q)\neq 0,w_2(z)=(z-q)^{w+1}h(z)$ with some holomorphic 
function $h(z)$ satisfying $h(q)\neq 0$. 
Recall that a necessary condition for $q_j$ to be an apparent singularity of weight $w$ 
is that the roots of the indicial equation 
$$
	\rho(\rho-1) + G^{q_j}_0 \rho + H^{q_j}_0=0 
$$
be $0,w+1$, i.e. that $G^{q_j}_0=-w,H^{q_j}_0=0$. 
By Fuchs' relation and (\ref{sumw}), we must then have 
\begin{equation}\label{Fuchs}
	\sum_{i=0}^n (\rho^i_1+\rho^i_2) = 1. 
\end{equation}
However, fixing the exponents is not sufficient for the singularity to be apparent: we also need to prescribe 
that the solution $w_1$ corresponding to the exponent $0$ do not contain a logarithmic term of the form 
$$
	C^{q_j}w_2(z)\ln(z-q_j) 
$$
where $w_2$ is a non-zero particular solution corresponding to the exponent $w+1$, 
i.e. we have to make sure that in this formula $C^{q_j}=0$. These add up to the three conditions 
\begin{align}
	G^{q_j}_0 &= -1,\label{condgj}\\
	H^{q_j}_0 &= 0,\label{condhj}\\
	C^{q_j} &= 0 \label{condcj}
\end{align}
which are necessary and sufficient for $q_j$ to be an apparent singularity of weight $w$. 

Notice that in total we have $6n-7$ indeterminates 
$$
	G_0,\ldots,G_{2n-3},H_0,\ldots,H_{4n-6}, 
$$
and so far we have written down $2(n+1)+3N$ conditions on them. 
In the case where all weights are equal to $1$ this latter number is $5n-4$, and 
in \cite{dubma} it is shown that symplectic conjugate coordinates associated to the 
variables $q_1,\ldots,q_{n-2}$ are (up to some constants) given by the values of 
\begin{equation}\label{hj1}
	H^{q_j}_1 = \lim_{z\to q_j}\left((z-q_j)\frac{H(z)}{\psi(z)^2}\right)
\end{equation}
(notice that this makes sense because of (\ref{condhj})). 
Hence, it is natural to impose that these coordinates take some prescribed values $p_j$ as well: 
\begin{equation}\label{condpj}
	H^{q_j}_1 = p_j. 
\end{equation}

Equations (\ref{condpj}) for $j\in\{1,\ldots,n-2\}$ with the previous ones add up to 
a total of $6n-6$ conditions in $6n-7$ indeterminates, so in principle this system 
could be overdetermined; however, the relation (\ref{Fuchs}) means that at least one 
of the conditions is redundant. 

In the case where some $w_j\geq 2$ we slightly have to modify the statements of the 
previous paragraph: the quantities (\ref{hj1}) still constitute some of the dual 
coordinates, but others have to be included too. Analogously to (\ref{hj},\ref{hj1}) 
define $H^{q_j}_l$ to be the $l$'th term of the Taylor-expansion of $(z-q_j)^2H/\psi^2$ 
at $q_j$. 

The first result of this paper is 
\begin{thm}\label{thm}
Conditions (\ref{condgi}, \ref{condhi}) for $i\in\{0,\ldots,n\}$ and 
(\ref{condgj},\ref{condhj},\ref{condcj},\ref{condpj}) for $j\in\{1,\ldots,N\}$ form 
a system of independent conditions except for a linear combination of 
(\ref{condgi},\ref{condgj}) implied by (\ref{Fuchs}). 
Furthermore, in the case $q_j\notin P$ for all $j$ if one fixes in addition the quantities 
$H^{q_j}_2,\ldots,H^{q_j}_{w_j}$ arbitrarily for all $j\in\{1,\ldots,N\}$ then (assuming (\ref{Fuchs})) 
there exists a unique set of coefficients $G_0,\ldots,G_{2n-3},H_0,\ldots,H_{4n-6}$ satisfying the 
above conditions. 
Finally, in the case where some $q_j=t_i\in P$, the same statement holds with
$H^{q_j}_{2w_j + 1},\ldots,H^{q_j}_{3w_j}$ fixed instead of (\ref{condpj}). 
\end{thm}

This theorem will be shown in Section \ref{sec:thm}.
As a consequence of the theorem, in Section \ref{sec:cor} we will also show 
\begin{cor}\label{cor}
For any finite multiset $Q\subset \C\setminus P$ of arbitrary total weight the space of second-order 
linear differential equations whose eigenvalues of residues at $t_i$ are $\rho^i_1,\rho^i_2$ and that have 
apparent singularities of weight $w_j$ at $q_j\in Q$ is either empty or an
affine space of dimension $n-2$. 
\end{cor}

Finally, in Section \ref{sec:geomint} we deduce the following geometric consequence. 
Let 
$$
        \sigma: \widetilde{K_{\P}(P)} \to K_{\P}(P)
$$
be the blow-up of the total space of the line bundle $K_{\P}(P)$ at the points 
corresponding to $(z-t_i)^{-1}\rho^i_1,(z-t_i)^{-1}\rho^i_2\in T^*_{t_i}\P$. 
Furthermore, let us denote by $F_{t_i}$ the proper transform of $T^*_{t_i}\P$ 
with respect to $\sigma$ and set 
$$
        X = \widetilde{K_{\P}(P)} \setminus \cup_{i=0}^n F_{t_i}. 
$$
Clearly, $X$ inherits from $K$ a projection to $\P$, and we continue to denote the 
image of this map by $q$. 
Consider the Hilbert scheme $X^{[n-2]}$ of $n-2$ points on $X$; it is known to be a smooth 
variety \cite[Chapter 1]{nak}. 

On the other hand, let $\Mod$ be the moduli space of parabolically stable logarithmic connections 
of rank $2$ and degree $-1$ with given singular set $P$ and given generic eigenvalues of the residue 
$\rho^i_1,\rho^i_2$ (satisfying (\ref{Fuchs})) and $\Mod^0$ be the open subset consisting of 
connections whose underlying holomorphic vector bundle isomorphic to 
$\mathcal{O} \oplus \mathcal{O}(-1)$. 

\begin{cor}\label{cor:geomint}
There exists a birational map from $\Mod^0$ into $X^{[n-2]}$. 
In particular, there exists a rational map $p_1$ from $\Mod^0$ to the $(n-2)$-th symmetric product 
$S^{n-2}\P$ of $\P$.
\end{cor}

For more details, in particular for a description of the image of this map, see Section \ref{sec:geomint}. 

Notice that a similar (but less detailed) relationship between moduli spaces of Higgs bundles on a 
compact curve and the Hilbert scheme of points on the total space of the cotangent bundle of the curve 
appears in \cite[Section 7.3]{nak}. 
It is an intriguing question whether the map of Corollary \ref{cor:geomint} can be defined 
on a larger subset of $\Mod^0$ (possibly $\Mod^0$ itself) and if so whether its image can be identified. 
Unfortunately, for the moment we haven't got a sufficiently deep understanding of elementary 
transformations of logarithmic connections at the punctures to be able to answer this question. 
We actually suspect that the map extends to all of $\Mod$ to yield a birational map to $X^{[n-2]}$. 
If this is the case, it would be interesting to understand its image. 
Notice also that the special case $n=3$ (Painlev\'e VI) is treated for example in \cite{iis}; 
in that case $n-2=1$ and $\Mod=X^{[1]}=X$. 
On the other hand, in \cite{ls} the case $n=4$ will be studied; in that case one 
has to deal with $X^{[2]}$ which is known to be the blow-up of $S^2(X)$ along the diagonal $\Delta$. 
In particular, the component of $\pi^{-1}(q,p)$ of maximal dimension is just $1$-dimensional; 
the open subset we use is in fact easily seen to be isomorphic to $\C$, and in $\pi^{-1}(q,p)$ 
one has an additional point that compactifies this $\C$ into a $\P$.

\section{The total weight $n-2$ case}\label{sec:thm}
For the sake of simplicity, we will only perform the detailed computations in the weight 
$1$ case (i.e. all $w_j=1$ for $j\in\{1,\ldots,n-2\}$) and content ourselves by providing 
the result of the same arguments in the general case. 
Clearly, conditions (\ref{condgi},\ref{condhi},\ref{condgj},\ref{condhj}) for $i\geq 1$ 
in terms of the indeterminates 
$$
	G_0,\ldots,G_{2n-3},H_0,\ldots,H_{4n-6}
$$ 
read 
\begin{align}
	G_0 + G_1 t_i + \cdots + G_{2n-3} t_i^{2n-3} &= \alpha_i \label{condgi2}\\
	H_0 + H_1 t_i + \cdots + H_{4n-6} t_i^{4n-6} &= \beta_i \label{condhi2}\\
	G_0 + G_1 q_j + \cdots + G_{2n-3} q_j^{2n-3} &= \gamma_j \label{condgj2}\\
	H_0 + H_1 q_j + \cdots + H_{4n-6} q_j^{4n-6} &= 0 \label{condhj2}
\end{align}
for appropriate constants $\alpha_i,\beta_i,\gamma_j$ only depending on $P,Q,i,\rho^i_1,\rho^i_2$ or 
$P,Q,j$ respectively (that we do not make explicit). 
The number of conditions (\ref{condgi2},\ref{condgj2}) involving the 
coefficients $G_m$ is $(n+1)+(n-2)=2n-1$. We have $2n-1$ linear equations in $2n-2$ indeterminates. 
However, as we already mentioned, because of (\ref{Fuchs}) and (\ref{gj}) one of the equations 
(say the one for $G^{t_0}_0$) is redundant. 
There remains an equal number of equations as variables, and the matrix of the system is a 
Vandermonde matrix with parameters $t_1,\ldots,t_n,q_1,\ldots,q_{n-2}$. These parameters are 
pairwise different by assumption, so the coefficients $G_0,\ldots,G_{2n-3}$ are uniquely determined. 

We now come to equations (\ref{condpj}). We clearly have 
\begin{align*}
	H^{q_j}_1 &= \frac{\d}{\d z}\left((z-q_j)^2\frac{H(z)}{\psi(z)^2}\right)|_{z=q_j}\\
		&= \kappa_j H(q_j) + \mu_j \left(\frac{\d H(z)}{\d z}\right)|_{z=q_j} \\
		&= \mu_j \left(H_1 + 2 H_2 q_j + \cdots + (4n-6) H_{4n-6} q_j^{4n-7}\right)
\end{align*}
where the last equality holds because of (\ref{condhj2}), and where the constants 
\begin{align*}
	\kappa_j &= \frac{\d}{\d z}\left(\frac{(z-q_j)^2}{\psi(z)^2}\right)|_{z=q_j} \\
	\mu_j &= \frac{1}{\prod_{i=1}^n(q_j-t_i)^2\prod_{k\neq j}(q_j-q_k)^2} \neq 0
\end{align*}
only depend on $P,Q,j$. We infer that (\ref{condpj}) is equivalent to 
\begin{equation}\label{condpj2}
	H_1 + 2 H_2 q_j + \cdots + (4n-6) H_{4n-6} q_j^{4n-7} = \frac{p_j}{\mu_j}. 
\end{equation}

Finally, let us come to equations (\ref{condcj}). 
For this purpose, we first introduce the Laurent series near $q_j$
\begin{align}
	\frac{G(z)}{\psi(z)} &= -\frac{1}{z-q_j} + G^{q_j}_1 + G^{q_j}_2(z-q_j)+\cdots \\
	\frac{H(z)}{\psi(z)^2} &= \frac{H^{q_j}_1}{z-q_j} + H^{q_j}_2 + H^{q_j}_3(z-q_j)+\cdots \label{Laurenth}
\end{align}
where we have used (\ref{condgj},\ref{condhj},\ref{hj1}). 
\begin{lem}
The logarithm-freeness condition (\ref{condcj}) is equivalent in the weight $1$ case to 
$$
	\left(H^{q_j}_1 + G^{q_j}_1\right)H^{q_j}_1 + H^{q_j}_2 = 0
$$
and in the weight $2$ case to 
$$
   H^{q_j}_3 + 2(H^{q_j}_2 + 2 G^{q_j}_2)H^{q_j}_1 + \frac 1 2 (H^{q_j}_1 + 2 G^{q_j}_1)
   \left(H^{q_j}_2 + 2 (H^{q_j}_1 + G^{q_j}_1) H^{q_j}_1 \right) = 0. 
$$
\end{lem}
\begin{rk}\label{rk}
If the weight is $w>2$ then the explicit form of the relevant condition becomes 
more involved. What remains true however is that one obtains $H^{q_j}_{w+1}$ 
as a polynomial of the quantities $H^{q_j}_l$ and $G^{q_j}_l$ for $l\in \{2,\ldots,w\}$. 
In addition, the polynomial in question is weighted homogeneous if one assigns 
weight $l$ to $H^{q_j}_l$ and $G^{q_j}_l$. 
These facts can be easily proved by induction using the method of the proof below. 
\end{rk}
\begin{proof}
We are looking for the necessary and sufficient condition for the existence of an integral $w_1(z)$ in 
power series form near $q_j$
$$
	w_1(z) = a_0 + a_1(z-q_j) + a_2(z-q_j)^2+\cdots, 
$$
with $a_0\neq 0$. Writing down the consecutive terms of the expansion of the first two 
derivatives of $w_1$ and substituting them into (\ref{eq}) the lemma follows. 
\end{proof}

\subsection{Weight $1$ case ($q_j\notin P$)}
We can now write down explicitly conditions (\ref{condcj}). 
Indeed, as we have already seen, the coefficients $G_0,\ldots,G_{2n-3}$ are uniquely 
determined, hence so is $G^{q_j}_1$. 
On the other hand, by (\ref{condpj}) we have $H^{q_j}_1=p_j$. Hence, by the lemma we see that 
(\ref{condcj}) can be rewritten as 
\begin{equation*}
	H^{q_j}_2 = -p_j^2 - G^{q_j}_1 p_j. 
\end{equation*}
Now, just as above we have 
\begin{align*}
	2 H^{q_j}_2 &= \frac{\d^2}{\d z^2}\left((z-q_j)^2\frac{H(z)}{\psi(z)^2}\right)|_{z=q_j}\\
		&= 2\frac{\kappa_j}{\mu_j}p_j + \mu_j\left(\frac{\d^2 H(z)}{\d z^2}\right)|_{z=q_j}. 
\end{align*}
Comparing the last two formulae, we see that (\ref{condcj}) is equivalent to 
\begin{equation}\label{condcj2}
	2 H_2 + 6 H_3 q_j + \cdots + (4n-6) (4n-7) H_{4n-6} q_j^{4n-8} = \delta_j p_j^2 + \epsilon_j p_j
\end{equation}
for some $\delta_j\neq 0$ and $\epsilon_j\in\C$ only depending on $P,Q,j$. 
Now, the matrix of the system formed by the linear equation 
$$
	H_{4n-6} = \rho^0_1\rho^0_2 
$$
(which, by (\ref{g0h0}) is just (\ref{condhi}) for $i=0$) joint with 
(\ref{condhi2},\ref{condhj2},\ref{condpj2},\ref{condcj2}) in the indeterminates 
$H_0,\ldots,H_{4n-6}$ is the following confluent Vandermonde matrix 
$$
	\begin{pmatrix}
		0 & 0 & 0 & 0 & \cdots & 1\\
		1 & t_1 & t_1^2 & t_1^3 & \cdots & t_1^{4n-6} \\
		\vdots & \vdots & \vdots & \vdots & \ddots & \vdots \\
		1 & t_n & t_n^2 & t_n^3 & \cdots & t_n^{4n-6} \\
		1 & q_1 & q_1^2 & q_1^3 & \cdots & q_1^{4n-6} \\
		\vdots & \vdots & \vdots & \vdots & \ddots & \vdots \\
		1 & q_{n-2} & q_{n-2}^2 & q_{n-2}^3 & \cdots & q_{n-2}^{4n-6} \\
		0 & 1 & 2 q_1 & 3 q_1^2 & \cdots & (4n-6) q_1^{4n-7} \\
		\vdots & \vdots & \vdots & \vdots & \ddots & \vdots \\
		0 & 1 & 2 q_{n-2} & 3 q_{n-2}^2 & \cdots & (4n-6) q_{n-2}^{4n-7} \\
		0 & 0 & 2 & 6 q_1 & \cdots & (4n-6) (4n-7) q_1^{4n-8} \\ 
		\vdots & \vdots & \vdots & \vdots & \ddots & \vdots \\ 
		0 & 0 & 2 & 6 q_{n-2} & \cdots & (4n-6) (4n-7) q_{n-2}^{4n-8} 
	\end{pmatrix}.
$$
As it is well-known for instance in the theory of Hermite interpolation (and can be 
directly proved by differentiating 
a usual Vandermonde matrix), the determinant of this matrix is up to a nonzero constant equal to 
$$
	\prod_{1\leq i<k\leq n} (t_i-t_k)
	\prod_{i,j} (t_i-q_j)^3
	\prod_{1\leq j<l\leq n-2} (q_j-q_l)^9.
$$
As we assumed the points $t_1,\ldots,t_n,q_1,\ldots,q_{n-2}$ to be pairwise distinct the above expression 
is non-vanishing, and Theorem \ref{thm} is proved in the case where the weight of any 
$q_j$ is $1$ and $Q$ is disjoint from $P$. 

\subsection{Higher weight case ($q_j\notin P$)}
In the case where all $q_j\notin P$ but the weights of some of the $q_j$ are $w_j\geq 2$ 
with total weight $\sum_{j=1}^N w_j = n - 2$, in addition to $H^{q_j}_1$ the quantities 
$H^{q_j}_2,\ldots,H^{q_j}_{w_j}$ may also be chosen arbitrarily, however $H^{q_j}_{w+1}$ 
is fixed by these choices. 
Then, the matrix of the corresponding system is a confluent Vandermonde matrix in 
which the multiplicity of $q_j$ is equal to $w_j + 2$. 
As the $q_j$ are distinct from each other and from $t_1,\ldots,t_n$ we again deduce 
in the same way as above that the corresponding determinant is non-zero. 
Details are left to the reader. 

\subsection{General case}\label{subsec:thmgen}
Finally, consider the general case where one of the points $q_j$ of weight $w_j$ is 
equal to some of the $t_i$. Then the highest power of $(z-t_i)$ that divides $\psi$ is 
$w_j+1$ so instead of the $3w_j+2$ equations 
(\ref{condgi},\ref{condhi},\ref{condgj},\ref{condhj},\ref{condcj}) we obtain 
\begin{align*}
     G(t_i) & = \cdots = G^{(w_j-1)}(t_i) = 0, G^{(w_j)}(t_i) = -w_j!(\rho^i_1+\rho^i_2) \\
     H(t_i) & = \cdots = H^{(2w_j-1)}(t_i) = 0, H^{(2w_j)}(t_i) = (2w_j)!\rho^i_1\rho^i_2. 
\end{align*}
It follows that in the confluent Vandermonde matrix governing the system of equations 
the rows corresponding to the conditions 
(\ref{condgi},\ref{condhi},\ref{condgj},\ref{condhj},\ref{condcj}) must be replaced by rows corresponding 
to the first $w_j+1$ derivatives of $G$ at $t_i$ and the first $2w_j+1$ derivatives of $H$. 
As before, this leaves no freedom for the coefficients of $G$. 
Now instead of (\ref{condpj}) a natural set of additional parameters can be picked by prescribing 
$w_j$ further derivatives of $H$ at $t_i$ 
--- namely, $H^{(2w_j+1)}(t_i),\ldots,H^{(3w_j)}(t_i)$ ---, 
and the confluent Vandermonde matrix obtained this way 
still has nonvanishing determinant. This then finishes the proof of the genaral case of Theorem \ref{thm}. 

\section{The dimension of equations}\label{sec:cor}
Let us come to the case where the number $N$ of apparent singularities $Q$ is allowed to 
be arbitrary: $Q=\{q_1,\ldots,q_N\}$. Following the above proof it is then easy to check that the linear 
system governing the existence of a differential equation determines all coefficients $G_m$ uniquely and 
leads to a matrix similar to the case $n-2$ above on the $H_m$, except that the indices of the $q_j$ go 
from $1$ to $N$ and that the powers of $t_i$ range from $0$ to $2(n+N-1)$. 
Hence, the number of columns changes to $2n+2N-1$, whereas the number of rows becomes $n+3N+1$. 

Obviously, as already mentioned in the introduction, if $N<n-2$ then the corresponding linear system is 
still of maximal rank, because it is a subsystem of the system of the $n-2$ case. 
As the number of rows in this case is smaller than the number of columns, this shows that the 
system is underdetermined and in addition to $p_1,\ldots,p_N$, further $n-2-N$ of the parameters $H_m$ 
(say $H_{n+3N+1},\ldots,H_{2(n+N-1)}$) can be chosen arbitrarily. 
Thus, we obtain a space of parameters for the differential equations of dimension $n-2$. 

If, on the other hand, we have $N>n-2$, then the system is overdetermined. 
In what follows we will point out that by choosing the $p_j$ in a suitable way we can still ensure 
existence of the differential equation with singularities at $Q$. 
For this purpose, we first observe that just as in the rank $n-2$ case one can show that the first 
$2n+2N-1$ rows of the relevant system are still linearly independent. 
Next, notice that the number $N-n+2$ of remaining rows is always smaller than or equal to $N$. 
It follows that $N-n+2$ rows of the form 
\begin{equation}\label{lincombj}
	\begin{pmatrix}
	0 & 0 & 2 & 6 q_j & \cdots & (4n-6) (4n-7) q_j^{4n-8}
	\end{pmatrix},
\end{equation}
say for $j\in\{n-1,\ldots,N\}$, can be expressed as linear combinations of the first $2n+2N-1$ rows. 
Clearly, the coefficients of these linear combinations only depend on $P,Q$. 
Therefore, if the constants $\delta_j p_j^2 + \epsilon_j p_j$ on the right hand side of 
(\ref{condcj2}) for $j\in\{n-1,\ldots,N\}$ are linear combinations with the same coefficients 
of the constants on the right hand sides of the first $2n+2N-1$ equations, 
then there exists a solution. In order that a solution exist, the constants $p_1,\ldots, p_N$ must 
therefore fullfill a system of $N-n+2$ quadratic equations. These equations are clearly independent 
from each other, since the one coming from the linear combination of (\ref{lincombj}) for a fixed 
$j\in\{n-1,\ldots,N\}$ contains the quadratic term $\delta_j p_j^2$ (recall from (\ref{condcj2}) 
that $\delta_j\neq 0$), while the ones coming from the linear combination of the similar rows for 
$k\neq j$ do not contain such a quadratic term. 
We conclude that for the existence of a differential equation of the required positions $Q$ of apparent 
singularities and required momenta $p_1,\ldots, p_N$ the latter are constrained by $N-n+2$ independent 
conditions; whence a total number $n-2$ of parameters for the differential equations with apparent 
singularities at $Q$ in this case too. This concludes Corollary \ref{cor}. 

\section{Geometric interpretation}\label{sec:geomint}

In this section our aim is to show Corollary \ref{cor:geomint}.
Denote by 
$$
   \pi: X^{[n-2]} \to S^{n-2}(X)
$$
the Hilbert-Chow morphism from the Hilbert scheme of $n-2$ points on $X$ to the $(n-2)$-th 
symmetrix power of $X$. 
Recall that $\pi$ is an isomorphism away from the big diagonal $\Delta\subset S^{n-2}(X)$ 
(the complement of the set of $n-2$ distinct points). 
The map $q:X\to \P$ induces a map $q:S^{n-2}(X)\to S^{n-2}(\P)$ and hence a map 
$q\circ\pi:X^{[n-2]} \to S^{n-2}(\P)$. 
On the other hand, for each $N\leq n-2$ and $N$-tuple of distinct points 
$(q_1,p_1),\ldots,(q_N,p_N)\in X$ 
with multiplicities $w_1,\ldots,w_N$ summing up to $n-2$ let us set 
$$
   (\vec{q},\vec{p}) = \sum_{j=1}^N w_j (q_j,p_j) \in S^{n-2}(X).
$$
Then there exists an algebraic isomorphism 
$$
   \pi^{-1} (\vec{q},\vec{p}) \cong \pi^{-1}(w_1 (q_1,p_1)) \times \cdots \times \pi^{-1}(w_N (q_N,p_N)) 
$$
where the letters $\pi$ on the right hand side refer to the Hilbert-Chow
morphisms of $X^{[w_1]},\ldots,X^{[w_N]}$ respectively. The map (in the reverse
direction) is given by 
\begin{equation}\label{idealmap}
   (I_1,\ldots ,I_N)\mapsto I = I_1\cap\cdots\cap I_N.
\end{equation}
Furthermore, for every $j$ there exists an open subset $C(q_j,p_j)$ of 
$\pi^{-1}(w_j (q_j,p_j))$ of maximal dimension $w_j-1$ consisting of ideals of the form 
$$
    I_j = \langle (q-q_j)^{w_j}, (p-p_j) + P(q-q_j) \rangle
$$
where $P$ is any polynomial of degree $w_j-1$ with vanishing constant term (so that $I_j$ depends on 
the $w_j-1$ coefficients of $P$). Now define the subset $U$ of $X^{[n-2]}$ by 
$$
   U = \pi^{-1}(S^{n-2}(X)\setminus \Delta) \cup \bigcup_{(\vec{q},\vec{p})} 
   (C(q_1,p_1) \times \cdots \times C(q_N,p_N)). 
$$
In other words we let $U$ contain the open stratum $S^{(n-2)}_{1,\ldots,1} = S^{n-2}(X)\setminus\Delta$ 
of $X^{[n-2]}$ and of all other strata $S^{(n-2)}_{w_1,\ldots,w_N}$ the open subset of the irreducible 
component of maximal dimension described by the above equations, that is for all $(\vec{q},\vec{p})$ 
the image of $C(q_1,p_1)\times\cdots\times C(q_N,p_N)$ by the map (\ref{idealmap}).
Finally, let us consider the open subset $U^0$ of $U$ defined as 
$$
   U^0 = U \cap (q\circ\pi)^{-1} S^{n-2}(\C\setminus P).
$$
\begin{rk}
Of course, since we restrict to $U^0$, we could have just as well started out with the open surface 
$K_{\P\setminus P}$ (total space of the same line bundle on the open curve) instead of $X$; 
nonetheless we choose to work with $X$ instead as we believe that the generalisations of 
the results of the present paper to the case where some of the apparent singularities might 
agree with real singular points will involve the surface $X$. 
This is suggested by the role played by $X$ in the $2$-dimensional Painlev\'e VI case \cite{iis}, 
see also our Concluding Remarks \ref{subsec:conclusion}. 
\end{rk}

On the other hand, let $\mathcal{M}$ denote the moduli space of stable parabolic logarithmic integrable 
connections $(E,\nabla,\{l^i\}_{i\in\{0,\ldots,n\}})$ of rank $2$ on $\P$, with logarithmic singularities 
at the points of $P$ with eigenvalues of the residue at $t_i$ equal to $\rho^i_1,\rho^i_2$ and arbitrarily 
chosen parabolic weights, see \cite{iis}. Here $l^i\subset E|_{t_i}$ is the one-dimensional eigenspace 
of $\res_{t_i}(\nabla)$ corresponding to $\rho^i_1$.
Instead of explaining the meaning of stability here, let us simply notice that 
because of our assumption of genericity of the eigenvalues, all logarithmic 
connections with these eigenvalues of its residues at $P$ are simple hence 
automatically parabolically stable (there are no nontrivial sub-bundles at all that could 
possibly violate the inequality between the slopes).

Let us now define a rational map from $\Mod^0$ to $U^0$ as follows \cite{ss};
the construction generalizes that of \cite{dubma}. 
Consider the first basis vector $e_1$ of $E\cong \mathcal{O} \oplus\mathcal{O}(-1)$ 
as cyclic vector of $\nabla$ on a Zariski open subset of $\P$ and set 
$f_2 \frac{dz}{\psi}=\nabla(e_1)$. 
Let $F$ be the locally free subsheaf of $E$ generated by the local sections
$e_1,f_2$. Then the torsion sheaf $T=E/F$ is of total length $n-2$. 
It defines points $q_1,\ldots,q_N$ of multiplicities $w_1,\ldots,w_N$
satisfying $w_1 + \cdots + w_N = n-2$. Assume $q_j\notin P$ for every
$j\in\{1,\ldots,N\}$; denote by $\Mod^1$ the set of connections satisfying 
this latter property. Then for an element $(E,\nabla)$ of $\Mod^1$ the 
connection matrix with respect to $e_1,f_2$ may be written as 
\begin{equation}\label{localform}
   d - \begin{pmatrix}
         0 &  a_{1,2}(z)\\
         1 & a_{2,2}(z)
     \end{pmatrix} \frac{dz}{\psi}
\end{equation}
for some rational functions $a_{1,2},a_{2,2}$ having poles at the points $q_j$. 
\begin{lem}\label{lem:localform}
In a neighborhood of $q_j$ we have 
\begin{align*}
   a_{2,2}(z) & = \frac{w_j}{z-q_j} + \mbox{holomorphic terms},\\
   a_{1,2}(z) & = \frac{p_j^0}{z-q_j} + p_j^1 + \cdots + p_1^{w_1-1}
   (z-q_j)^{w_j-2} + \mbox{higher order terms}.
\end{align*}
\end{lem}
\begin{proof}
Since $\nabla:E\to E\otimes K_{\P}(P)$ has no singularity at $q_j$, a sequence
of $w_j$ upper elementary modifications of the connection form with respect to some directions
has to produce a matrix whose entries are regular functions at $q_j$. 
Recall that upper elementary modification at $q_j$ with respect to the direction $(p,1)^t$ is by definition 
the change of trivialization with respect to the matrix
$$
   \begin{pmatrix}
     1 & p\\
     0 & 1
     \end{pmatrix}
   \begin{pmatrix}
     1 & 0 \\
     0 & (z-q_j)^{-1}
     \end{pmatrix}
   \begin{pmatrix}
     1 & -p\\
     0 & 1
     \end{pmatrix}.
$$
On the diagonal entry $a_{2,2}$ it acts by adding $(z-q_j)^{-1}+O(1)$; this proves the first statement. 
The same argument also shows that the directions with respect to which the modifications at $q_j$ have to 
be applied are different from $e_1$; otherwise the elementary modification would introduce a term 
$(z-q_j)^{-1}$ in the $(1,1)$-entry of the matrix that no other elementary modification could cancel
(except for a lower one in the same direction). 
On the off-diagonal entry $a_{1,2}$ the elementary modification at $q_j$ with respect to the direction $(p,1)^t$ 
acts by
$$
   a_{1,2} \mapsto \frac{a_{1,2} - p (p+ a_{2,2})}{z-q_j} + O(1)
$$
whence the second statement. 
\end{proof}

With the notations of Lemma \ref{lem:localform} set now 
$$
   I_j = \langle (q-q_j)^{w_j}, (p-p_j^0) - p_j^1 (q-q_j) -\cdots - p_1^{w_1-1}
   (q-q_j)^{w_j-1}  \rangle.
$$
Then 
$$
   I = I_1 \cap \cdots \cap I_N
$$
defines a point in $U^0$. This gives us the map of Corollary \ref{cor:geomint}. 
The procedure associating $(F,a_{1,2},a_{2,2})$ to $(E,\nabla)$ is clearly algebraic 
in $(E,\nabla)$ as it depends on no choice (the cyclic vector $e_1$ being fixed). 
So the mapping defined above is algebraic wherever defined. 
We need to check that its domain $\Mod^1$ is open and its image is indeed equal to $U^0$. 
Openness follows from equality of the dimensions and the fact that the image contains $U^0$. 
So we need to identify the image. 

It follows from Theorem \ref{thm} that given the following data there exists 
a unique Fuchsian differential equation with exponents $\rho^i_1,\rho^i_2$ 
at $t_i$ and having apparent singularities at the points $q_1,\ldots,q_N$ with auxiliary 
parameters $p_j^k$: 
\begin{itemize}
\item In case $q_1,\ldots,q_N\in \P\setminus P$ the parameters are 
\begin{equation}\label{parameters}
   (q_1, p_1^0, \ldots ,p_1^{w_1-1}), \ldots, (q_N, p_N^0, \ldots ,p_N^{w_N-1})
\end{equation}
where $q_1,\ldots,q_N$ are the apparent singularities of weight $w_1,\ldots,w_N$ respectively 
(satisfying (\ref{sumw})), and $p_j^l = H^{q_j}_{l+1}$ is the $(l+1)$'th term in the Laurent expansion of 
$H$ near $q_j$, for $j\in\{1,\ldots,N\},l\in\{0,\ldots,w_j-1\}$. 
\item In the case where some $q_j = t_i\in P$ of weight $w_j$, corresponding 
fiber parameters are given by 
$$
   p_j^0 = H^{t_i}_{2w_j + 1}, \ldots, p_j^{w_j-1} = H^{t_i}_{3w_j}.
$$ 
\end{itemize}

Therefore, in order to show that the map as defined on the open subset is onto $U^0$ we need to check that the 
underlying vector bundle of the connection produced by Theorem \ref{thm} out of the data
(\ref{parameters}) is indeed of Birkhoff type $(0,-1)$. 

\subsection{Weight $1$ case}\label{wt1}

Assume first that in the above expressions $w_j=1$ for all $j\in\{1,\ldots,n-2\}$. 
Using the notation 
\begin{align}
	\psi_T(z) & = \prod_{j=1}^n(z-t_j)\label{psit}\\
	\psi_Q(z) & = \prod_{\alpha=1}^{n-2}(z-q_j)\label{psiq}, 
\end{align}
so that $\psi(z)=\psi_T(z)\psi_Q(z)$, consider the local frame on the open chart 
$\P\setminus \{t_0\}$ away from infinity given by 
\begin{align}
    w_1 & = w \notag \\
    w_2 & = \psi_T \frac{\d w_1}{\d z}. \label{extension} 
\end{align}
In this frame, equation (\ref{eq}) is equivalent to the connection $\d + A$ where the connection matrix
$A$ is the negative of a companion matrix 
\begin{equation}\label{a}
    A  = \frac{\d z}{\psi_T(z)} \begin{pmatrix} 
                                           0 &  H(z)/\psi_Q^2(z) \\
                                            - 1  & G(z)/\psi_Q(z)-\psi_T'(z)
                           \end{pmatrix}
\end{equation}
(notice that this is logarithmic at $q_j$ because of (\ref{condhj})). 
Observe also that intrinsically $H{\d z}/{\psi}$ has to be considered as a section of 
$K_{\P}(P)$ where $K_{\P}$ stands for the canonical bundle of $\P$; in particular, 
$H^{q_j}_1$ is naturally an element of $T_{q_j}^*\P$ (see also \cite{ss}). 
One can introduce a similar frame near $t_0$ with local coordinate $\zeta=z^{-1}$: 
\begin{align}
    \tilde{w}_1 & = w \notag \\
    \tilde{w}_2 & = \zeta \frac{\d w_1}{\d \zeta} \label{extinf} 
\end{align}
and check that equation (\ref{eq}) is equivalent to a connection with logarithmic singularity at $\zeta = 0$ 
and connection matrix in form similar to (\ref{a}) with respect to this frame too. 
Furthermore, as it is easy to see \cite{sz} the frames (\ref{extension},\ref{extinf}) 
match together to give a holomorphic bundle 
\begin{equation}\label{f-birkhoff}
        F \cong \mathcal{O} \oplus T_{\P}(-P),
\end{equation}
where $T_{\P}$ stands for the tangent bundle of $\P$; in particular, 
$F$ is of Birkhoff type $(0,1-n)$. 
At $t_i\in P$ the residue of $-A$ is the matrix 
$$
        \res_{t_i}(-A) = 
        \begin{pmatrix}
                0 &  -\rho^i_1\rho^i_2  \\
               1 & \rho^i_1 + \rho^i_2
        \end{pmatrix},
$$
whose eigenvalues are indeed given by $\rho^i_1,\rho^i_2$, with corresponding 
eigenspaces spanned by 
$$
       \begin{pmatrix}
                -\rho^i_2 \\ 1
       \end{pmatrix},
       \begin{pmatrix}
                -\rho^i_1 \\ 1
       \end{pmatrix}.
$$
At the points of $Q$ the eigenvalues of the residue are equal to $0,1$ (see (\ref{condgj})). 
The residue is given by 
\begin{equation}\label{resqi}
        \res_{q_j}(A) = 
        \begin{pmatrix}
                0 & -p_j \\
                0 & 1
        \end{pmatrix},
\end{equation}
whose $1$-eigenvector is $(-p_j,1)^T$. 
The extra conditions (\ref{condcj}) are equivalent to saying that the logarithmic 
singularities at the points of $Q$ can be removed using an upper elementary 
transformation in the direction of the $1$-eigenspace of the residue of $A$.
We have a total number of $n-2$ apparent singularities, 
and each elementary transformation is performed in the direction of the line spanned by 
$(-p_j,1)^T$ with respect to the frame dual to (\ref{extension}), hence it 
increases the second number in the type of $F$. 
It follows that after the elementary transformations at all $q_j$ we obtain a bundle 
$$
        E\cong \mathcal{O} \oplus \mathcal{O}(-1)
$$
which inherits a logarithmic connection $\nabla$ from $\d - A$ with logarithmic 
singularities only at $P$, with eigenvalues $\rho^i_1,\rho^i_2$ of its residue at $t_i$. 
In other words, the type of $E$ is $(0,-1)$. 
Notice also that $E$ admits a unique global section $(1,0)$ 
(up to a scalar multiple), and choosing this vector as cyclic vector gives back 
the apparent singularity set $Q$. 

\subsection{Higher weight case}

This is completely analogous to the previous case. According to Lemma \ref{lem:localform} 
the directions of the sucesssive elementary transforms one needs to apply to (\ref{localform}) 
in order to make its singularities at $q_1,\ldots,q_N$ disappear are all complementary 
to the line spanned by $e_1$. Just as in the weight $1$ case this implies that the underlying 
vector bundle $E$ is of the desired Birkhoff type $(0,-1)$. This finishes the proof of 
Theorem \ref{thm}. 

\section{Parabolic structure and apparent singularities}

In this section we apply our methods to rederive the following result of 
Loray, Saito and Simpson. Let $n=3$ (so $N=1$), fix the eigenvalues 
$\rho^i_1$ generically so that their sum is equal to $1$ and consider the map 
$$
   q:\Mod^0\to \P\setminus P
$$
introduced in Corollary \ref{cor:geomint}. On the other hand, there is a natural forgetful map 
$$
   Q:\Mod^0\to \P\setminus P,
$$
where $\P\setminus P$ is a Zariski open subset of the moduli space of parabolic bundles and
$p_2$ maps the equivalence class of the triple $(E,\nabla,\{l^i\}_{i\in\{0,\ldots,n\}})$ to 
the equivalence class of $(E,\{l^i\}_{i\in\{0,\ldots,n\}})$. 
\cite[Proposition 9.1]{lss} then states that the fibers of $q$ in $\Mod^0$ intersect the fibers 
of $Q$ transversely in a single point. 

Assume we are given two parabolic logarithmic connections $(E_1,\nabla_1),(E_2,\nabla_2)$ with 
the given eigenvalues $\rho^i_k$ ($i\in\{0,1,2,3\}, k\in\{1,2\}$) of the residue. 
The Birkhoff-type of $E_1$ and of $E_2$ is $(0,-1)$, so choosing $(1,0)^t$ as cyclic 
vector gives us $q$-coordinates $q_1$ for $E_1$ and $q_2$ for $E_2$, so that 
$(E_1,\nabla_1)$ reduces to companion-matrix form (\ref{eq}) (and the same holds for $(E_2,\nabla_2)$ too). 
In addition, it follows from Theorem \ref{thm} that the scalar equation corresponding to 
$(E_1,\nabla_1)$ is uniquely determined by the quantity (\ref{condpj}) (and the same for $(E_2,\nabla_2)$). 
Let $(q_1,p_1)$ and $(q_2,p_2)$ be the parameters so obtained of $(E_1,\nabla_1)$ and of $(E_2,\nabla_2)$ 
respectively. 

We know from Subection \ref{wt1} that the bundles $E_1,E_2$ can be obtained from 
$F=\mathcal{O}\oplus T_{\P}(-P)\cong\mathcal{O}\oplus\mathcal{O}(-2)$ (see \ref{f-birkhoff})
applying elementary transformations at the points $q_1,q_2$ with respect to the directions $p_1,p_2$
respectively. Let $\mathcal{E}lm_{q,-p}$ stand for elementary transformation of 
$F$ at the point $q\in\P$ in direction of the line spanned by $(-p,1)^t$ with respect to the 
standard trivialisation $(f_1,f_2)\in \mathcal{O}\oplus\mathcal{O}(-2)$ of $F$. 
Furthermore, it follows from \cite{sz2} that the parabolic structure on $F$ is determined by the 
eigenvalues $\rho^i_k$. Since these elementary modifications are applied away from the parabolic points, 
from this we deduce the equality of parabolic bundles 
\begin{equation}\label{flip}
   E_2 = \mathcal{E}lm_{q_2,-p_2} \circ \mathcal{E}lm_{q_1,-p_1}^{-1}(E_1).
\end{equation}
Positive and negative elementary modifications of holomorphic bundles are the simple building blocks 
of rational transformations of holomorphic bundles just as blow-ups and blow-downs are building blocks 
of birational maps between algebraic varieties. Therefore, in analogy with the language of birational 
geometry, we will name the transformation appearing on the right-hand side of (\ref{flip}) a \emph{flip}. 

The elementary transformation $\mathcal{E}lm_{q,-p}^{-1}$ is given by the following composition of changes of bases: 
\begin{align*}
   e_1' & = e_1, & e_2' = e_2-p e_1, \\
   e_1'' & = e_1' = e_1, & e_2'' = \frac{e_2'}{z-q} = \frac{e_2-p e_1}{z-q},
\end{align*}
hence in the chosen basis of $E$ we have the matrix forms
\begin{equation}\label{elmj}
   \mathcal{E}lm_{q_j,-p_j}^{-1} 
     = \begin{pmatrix}
       1 & -\frac{p_j}{z-q_j} \\
       0 & \frac{1}{z-q_j}
       \end{pmatrix}.
\end{equation}
It then follows that with respect to the same bases, 
\begin{equation}\label{flip-matrix}
    \mathcal{E}lm_{q_2,p_2} \circ \mathcal{E}lm_{q_1,p_1}^{-1} = 
       \begin{pmatrix}
              1 & \frac{p_2-p_1}{z-q_1}\\
              0 & \frac{z-q_2}{z-q_1}
         \end{pmatrix}
\end{equation}

Let us now fix $q_1,q_2\in\C\setminus P,p_1\in\C$ and look for $p_2\in\C$ 
such that $E_1$ and $E_2$ are parabolically isomorphic bundles. 
As $E_1\cong\mathcal{O}\oplus\mathcal{O}(-1)\cong E_2$, we have 
$$
   \mathcal{H}om(E_1,E_2) \cong \begin{pmatrix} 
                          \mathcal{O} & \mathcal{O}(-1) \\
                          \mathcal{O}(1) & \mathcal{O}
                       \end{pmatrix},
$$
so the space of global homomorphisms from $E_1$ to $E_2$ 
(not necessarily preserving the parabolic structure) is $4$-dimensional. 
However, the scalar homomorphisms act trivially on the parabolic structure. 
In conclusion, there is a $3$-dimensional space of global homomorphisms 
(not necessarily compatible with the parabolic structure). 
The question is whether it is possible to choose $p_2$ and 
$g\in H^0(\P,\mathcal{H}om(E_1,E_2))$ such that the 
action of (\ref{flip-matrix}) on the parabolic structure agree with that of $g$. 

Phrased differently we are then looking for $p_2,g$ such that for all $i\in\{0,1,2,3\}$ 
the line $l^i$ be an eigendirection of 
\begin{equation}\label{gflip}
   g(t_i)^{-1} \begin{pmatrix}
              1 & \frac{p_2-p_1}{t_i-q_1}\\
              0 & \frac{t_i-q_2}{t_i-q_1}
         \end{pmatrix}
\end{equation}
Let $l^i$ be the line spanned by $(1, u^i)^t$ (generically there exists a unique such $u^i$). Set 
\begin{equation}\label{g2}
   g^{-1}=\begin{pmatrix}
     a & 0 \\
      b + c z & 1
   \end{pmatrix},
\end{equation}
where the $(2,2)$-entry is normalized to $1$ using a constant gauge transformation (which does not 
modify the action on the parabolic structure). Then the linear map (\ref{gflip}) maps $l^i$ to 
\begin{equation}\label{gflipaction}
   \begin{pmatrix}
       a \left( 1 + u^i \frac{p_2-p_1}{t_i-q_1}\right) \\
       (b + c t_i) \left( 1 + u^i \frac{p_2-p_1}{t_i-q_1}\right) + u^i \frac{t_i-q_2}{t_i-q_1} 
   \end{pmatrix}
\end{equation}
Now this vector is a constant multiple of $(1, u^i)$ if and only if the following condition holds 
\begin{equation}\label{linearsystem}
   b + c t_i - a u^i = -u^i \frac{t_i-q_2}{t_i-q_1} \left( 1 + u^i \frac{p_2-p_1}{t_i-q_1}\right)^{-1}.
\end{equation}
The right hand side of this equation can be simplified into 
\begin{equation}\label{rhs}
   -u^i \frac{t_i-q_2}{t_i-q_1+u^i(p_2-p_1)}.
\end{equation}
These conditions for $i\in\{0,1,2,3\}$ represent $4$ quadratic equations in the unknowns $a,b,c,p_2$ of a very 
special type: namely, they are linear in $a,b,c$ and  $p_2$ only appears on the right hand side of the linear 
equations. The coefficient matrix of the linear system is 
$$
   \begin{pmatrix}
     -u_0 & 1 & t_0 \\
     -u_1 & 1 & t_1 \\
     -u_2 & 1 & t_2 \\
     -u_3 & 1 & t_3
   \end{pmatrix}
$$
which is of maximal rank $3$ for generic values of the parabolic parameters $u^i$. 
Let then $\alpha_1,\alpha_2,\alpha_3$ be coefficients of a linear combination of the second, third and 
fourth rows of this matrix yielding the first row: 
\begin{align}
  \alpha_1 u_1 + \alpha_2 u_2 + \alpha_3 u_3 & = u_0 \notag \\
  \alpha_1 + \alpha_2 + \alpha_3 & = 1 \label{linearcombination} \\
  \alpha_1 t_1 + \alpha_2 t_2 + \alpha_3 t_3 & = t_0.\notag
\end{align}
Then the system (\ref{linearsystem}) admits a solution if and only if the same linear combination 
of the expressions (\ref{rhs}) for $i\in\{1,2,3\}$ is equal to (\ref{rhs}) for $i=0$:
\begin{align*}
   \alpha_1 u^1 \frac{t_1-q_2}{t_1-q_1+u^1(p_2-p_1)} + &  
   \alpha_2 u^2 \frac{t_2-q_2}{t_2-q_1+u^2(p_2-p_1)} \\ 
   + \alpha_3 u^3 \frac{t_3-q_2}{t_3-q_1+u^3(p_2-p_1)} & =
    u^0 \frac{t_0-q_2}{t_0-q_1+u^0(p_2-p_1)}.
\end{align*}
After resolving the quotients the equation one obtains appears to be cubic in $p_2$. 
However, a computation shows that the coefficient of $p_2^3$ on the left hand side becomes 
$$
   -u^0u^1u^2u^3(\alpha_1 (t_1-q_1) + \alpha_2 (t_2-q_1) + \alpha_3 (t_3-q_1 ))
$$
which, using the equations (\ref{linearcombination}), is equal to 
$$
   -u^0u^1u^2u^3 (t_0-q_1).
$$
As this latter is equal to the coefficient of $p_2^3$ on the right hand side, 
it follows that the polynomial obtained after resolving the quotients is at most 
quadratic. 
A further direct (but tedious) computation shows that the polynomial is indeed quadratic 
for generic values of the parameters, therefore yielding $2$ solutions for $p_2,a,b,c$. 
Notice however that one of the solutions is due to the stacky structure of the space of 
points on $\P$ modulo M\"obius-transformations: indeed, for some values of the parameters 
$u^0,u^1,u^2,u^3$ there exists another configuration of points with the same cross-ratio, 
which however is not equivalent by M\"obius-transformations to $u^0,u^1,u^2,u^3$. 
For example, in case $q_2=q_1$ one of the solutions furnished by our method is the obvious solution 
$p_2=p_1,g=\mbox{Id}$, whereas the other one corresponds to the inequivalent quadruple with the same 
cross-ratio, corresponding to some value $p_2\neq p_1$ and $g\neq\mbox{Id}$. 
Said differently, although the action of $\mbox{PSl}_2(\C)$ is not transitive on quadruples of points on $\P$ 
but for most values of the cross-ratio it is simply transitive on quadruples sharing this given cross-ratio, 
its joint action (\ref{gflipaction}) with flips (\ref{flip}) is transitive but not simple.

\section{Concluding remarks}\label{subsec:conclusion}
Let us end this study by giving some insight into the case where some of the apparent singularities $q_j$ 
agree with some $t_i$; the situation is similar to the case of Painlev\'e VI, see \cite{iis}. 
We restrict here to the case where $w_j=1$. 
In this case the above interpretation of $p_j$ as the negative of the first coordinate 
of the $1$-eigendirection $(-p_j,1)^T$ of the residue of $-A^T$ 
(with respect to which elementary transformation has to be performed), breaks down. 
Indeed, in general the connection obtained by a negative elementary transformation 
applied to $\d - A$ will be logarithmic only if the direction with respect to which 
it is performed is one of the eigendirections $(-\rho^i_2, 1)^T$ or $(-\rho^i_1, 1)^T$ 
of $\res_{t_i}(- A)$. 
In different terms, the direction of the elementary transformation is prescribed; 
this is the geometric meaning of equations (\ref{condgi},\ref{condhi}). 
Instead of $H^{q_j}_1$, the free parameter this time is $H^{t_i}_3$ (see Subsection 
\ref{subsec:thmgen}). 
As $G$ is entirely determined, $H^{t_i}_3$ governs the first order behaviour 
of the exponents near $t_i$, that is the roots of 
$$
        \rho(\rho -1) + (G^{t_i}_0 + G^{t_i}_1(z-t_i)+O((z-t_i)^2)) \rho 
                  + (H^{t_i}_2 + H^{t_i}_3(z-t_i)+O((z-t_i)^2)) 
$$
are equal to 
$$
        \rho^i_{1,2}(z) = \rho^i_{1,2} + (H^{t_i}_3 + d^i_{1,2})(z-t_i) + O((z-t_i)^2),
$$
where $d^i_{1,2}\in\C$ only depends on $P,Q$ and the exponents. 
As the residue of the connection form $-A$ at $t_i$ is a companion matrix, 
it follows just as above that the corresponding eigendirections (with respect to which 
elementary transformation has to be carried out) are of the form 
$$
       \begin{pmatrix}
                -\rho^i_2(z) \\ 1
       \end{pmatrix},
       \begin{pmatrix}
                -\rho^i_1(z) \\ 1
       \end{pmatrix}.
$$
Said differently, instead of the value $p_j$ of the eigendirection, 
the free parameter in the case $q_j=t_i$ is its \emph{first-order behaviour}. 
In geometric terms, this amounts precisely to considering the blow-up 
$$
        \sigma: \widetilde{K_{\P}(P)} \to K_{\P}(P)
$$
of the total space of $K_{\P}(P)$ at the points $\rho^i_1,\rho^i_2\in T^*_{t_i}\P$. 
Let us define $F_{t_i}$ to be the proper transform of $T^*_{t_i}\P$ by $\sigma$ and set 
$$
        X = \widetilde{K_{\P}(P)} \setminus \cup_{i=0}^n F_{t_i}. 
$$
This gives evidence that the free parameters $(q_j,p_j)$ take value in $X$, and then by an 
argument similar to the one in the previous cases one should be able to describe the moduli 
space in terms of $X^{[n-2]}$. 
However, for a more detailed understanding of this description one would need to study thoroughly 
the effect of elementary modifications of logarithmic connections at real singular points; 
we leave this question for further study. 

\bibliographystyle{alpha}
\bibliography{independence}

\end{document}